\documentclass[11pt]{article}
\usepackage{amsmath, amsthm, amssymb}
\usepackage{geometry}
\usepackage{hyperref}
\usepackage{microtype}
\usepackage{enumitem}
\usepackage{tikz}
\usepackage{pgfplots}
\usepackage{booktabs, siunitx}
\pgfplotsset{compat=1.18}
\usepackage{graphicx}
\usepackage{float}
 
\newtheorem{theorem}{Theorem}

\theoremstyle{definition}

\title{\textbf{$L^p$ Asymptotics of the Möbius Energy Density of Helix Curves}}
\author{Yash Tiwari}
\date{}
 
\begin{document}
 
\maketitle

\begin{abstract}
Motivated by the recent work of Lipton \cite{lipton2026} on the M\"obius 
energy of helix curves, we extend the study to the $L^p$ asymptotics of 
the meromorphic family
\[
M_\rho(t) = \frac{\rho^2+1}{\rho^2 t^2 + 4 \sin^2(t/2)} - \frac{1}{t^2}.
\]
The helix has infinite M\"obius energy, but the arclength-rescaled energy 
density is finite. As $\rho \to 0$ the helix coils infinitely tight. 
Using contour integration and a careful Laurent expansion near the poles, 
we establish
\[
I_p(\rho) := \left(\int_{-\infty}^\infty M_\rho(t)^p \, dt\right)^{1/p} 
\sim C_p \, \rho^{-(2-1/p)}
\]
for integer $p > 1$, extended to real $p > 1$, where $C_p$ is an explicit 
constant involving $\zeta(2p-1)$. The result gives the precise $L^p$ 
blowup rate of the M\"obius energy density as the pitch $\rho \to 0$. 
The borderline case $p=1$ yields a logarithmic correction 
$I_1(\rho) \sim \log(1/\rho)/\rho$, recovering Lipton's main theorem. 
We derive a quantitative coiling barrier. Numerical verification confirms the 
scaling exponent to high precision.
\end{abstract}

\tableofcontents
\bigskip
 
%------------------------------------------------------------
\section{Setup and Main Result}
 
\subsection{Geometric Motivation}
 
This work extends the recent analysis of the M\"obius energy of helix 
curves carried out by Lipton \cite{lipton2026}. Let $j, p \geq 1$, $I$ be an interval or a circle and let $\gamma : I \to \mathbb{R}^n$ parametrize a $C^2$ curve. For $s\in I$, define the
pointwise energy, also referred to as the energy element or energy density, as
\[
E^{j,p}(\gamma,s) = \int_{I}\left( \frac{1}{|\gamma(t)-\gamma(s) |^j} - \frac{1}{D(\gamma(t) , \gamma(s))^j}\right)^p |\dot{\gamma(t)}|dt
\]
Where D denote the intrinsic distance along $\gamma$ , The main topic of our study is the $(j,p)$ O'Hara energy of $\gamma$ is given by
\[
E^{j,p}(\gamma,s) = \int_{I}E^{j,p}(\gamma,s)|\dot{\gamma(s)}|ds
\]
For our choice of study to extend the result for the helix parametrized with pitch $\rho$ ,  $H_{\rho}(t)$ = $(e^{it},\rho t)$. By rescaling according to the arclength element \cite{lipton2026} we have 
\[
E^{2,p} = \int_{-\infty}^{\infty} (M_{\rho}(t))^p dt
\]
The normalized $L^p$ energy would be 
\[
I_p(\rho) = (E^{2,p})^{\frac{1}{p}}=\left(\int_{-\infty}^{\infty} (M_{\rho}(t))^p dt \right)^{\frac{1}{p}}
\]
Our setup is to find the asymptotics for $\rho \to 0$. Which leads us to the question of energy blow up , further we give a precise blowup profile of the $L^p$ energy with the concentration of the pitch.
 
\subsection{The Function and Main Theorem}
 
Define the one-parameter family of functions $M_\rho : \mathbb{R} \to \mathbb{R}$ by
\begin{equation}\label{eq:Mrho}
  M_\rho(t) = \frac{\rho^2 + 1}{\rho^2 t^2 + 4\sin^2\!\tfrac{t}{2}} - \frac{1}{t^2}, \qquad \rho > 0.
\end{equation}
The apparent singularity at $t = 0$ is removable: expanding both terms in Taylor series shows $M_\rho(t) \to 0$ as $t \to 0$. For $t \neq 0$ the function is smooth and positive.
 
Our primary object of study is the $L^p$ norm
\begin{equation}\label{eq:Ip}
  I_p(\rho) = \left(\int_{-\infty}^{\infty} M_\rho(t)^p \, dt\right)^{1/p}, \qquad p > 1, \; p \in \mathbb{Z}.
\end{equation}
 
\begin{theorem}[Main Result]\label{thm:main}
For every  $p > 1$, as $\rho \to 0^+$,
\[
  I_p(\rho) \sim C_p \, \rho^{-(2 - 1/p)},
\]
where
\[
  C_p = \left(\binom{2p-2}{p-1}
\frac{1}{(4\pi)^{2p-2}}
\zeta(2p-1)\right)^{\frac{1}{p}}
\]
and $\zeta$ denotes the Riemann zeta function.
\end{theorem}
 
The proof occupies Sections~\ref{sec:poles}--\ref{sec:final}.
 
%------------------------------------------------------------
\section{Analytic Continuation and Pole Structure}\label{sec:poles}
 
Throughout Sections~2--5 we assume $p \in \mathbb{Z}$ with $p > 1$. 
The extension to real $p > 1$ is treated in Section~7. The poles of
\[
  f(z) := M_\rho(z)
\]
occur where the denominator $\rho^2 z^2 + 4\sin^2\!\tfrac{z}{2}$ vanishes.
 
\subsection{Finding the Poles}
 
At $\rho = 0$ the zeros of the denominator are exactly $z = 2n\pi$, $n \in \mathbb{Z}$. For small $\rho > 0$ we seek perturbed zeros of the form $z_n = 2n\pi + \varepsilon$ with $|\varepsilon| \ll 1$.
 
Substituting and expanding $\sin\!\tfrac{2n\pi + \varepsilon}{2} = \sin\!\tfrac{\varepsilon}{2}$:
\[
  \rho^2(2n\pi + \varepsilon)^2 + 4\sin^2\!\frac{\varepsilon}{2} \approx \rho^2 z_n^2 + \varepsilon^2 = 0,
\]
which gives $\varepsilon = \pm i\rho z_n$. To leading order $z_n \approx 2n\pi$, so
\begin{figure}[h]
    \centering
    \includegraphics[width=0.5\linewidth]{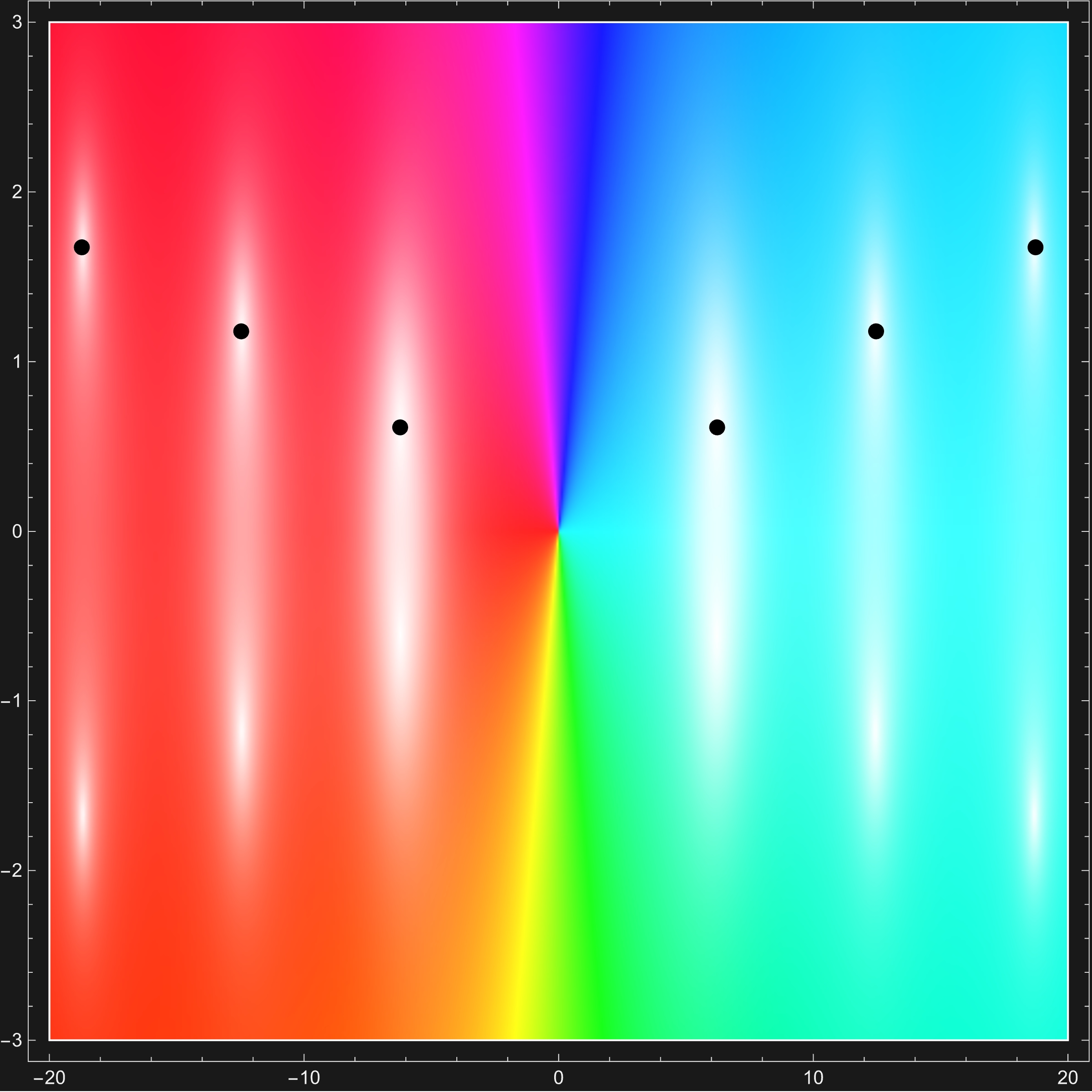}
    \caption{plot of $M_\rho(z)$ showing poles in the complex plane, 
located near $z_n = 2n\pi(1 \pm i\rho)$ for small $\rho$.}
    \label{fig:placeholder}
\end{figure}
\begin{equation}\label{eq:poles}
  z_n = 2n\pi(1 \pm i\rho), \qquad n \in \mathbb{Z} \setminus \{0\}.
\end{equation}

The poles with $\operatorname{Im}(z_n) > 0$ (upper half-plane) are $z_{n^+} = 2n\pi(1+i\rho)$ for $n > 0$ is the right branch of poles and $z_{n^-} = 2n\pi(1-i\rho)$ ,  $n < 0$ is the other branch.  We can make the sum symmetric by saying $z_{n^-} = - 2 n \pi(1-i \rho)$ for $n > 0$.  the contribution to the residue from these poles would be the same as the poles $z_{n^+}$.

The expansion is valid for $n\rho \ll 1$, which corresponds to the dominant 
range $n \lesssim 1/\rho$. Contributions from $n \gtrsim 1/\rho$ are 
controlled by the convergence of $\sum n^{-(2p-1)}$ for $p > 1$ and contribute 
at order $O(\rho^{2p-1})$, which is subleading. We will use this fact later 
in the proof.
 
%------------------------------------------------------------
\section{Contour Integration}\label{sec:contour}
 
\subsection{Reduction to Residues}
\begin{center}
\begin{tikzpicture}[scale=2]

% Axes
\draw[->] (-2,0) -- (2,0) node[right] {$\Re z$};
\draw[->] (0,-0.5) -- (0,2) node[above] {$\Im z$};

% Real axis segment
\draw[thick] (-1.5,0) -- (1.5,0);

% Arrow on real axis (left to right)
\draw[->, thick] (-0.8,0) -- (0.8,0);

% Semicircle
\draw[thick] (1.5,0) arc (0:180:1.5);

% Arrow on semicircle (counterclockwise)
\draw[->, thick] (0.7,1.32) arc (60:65:1.5);

% Labels
\node at (1.5,-0.2) {$R$};
\node at (-1.5,-0.2) {$-R$};
\node at (0.2,1.6) {$C_R$};

\end{tikzpicture}
\end{center}
 
Let $C_R$ denote the semicircular contour consisting of the segment $[-R,R]$ on the real axis and the arc $\gamma_2 : Re^{i\theta}$, $\theta \in [0,\pi]$. By the residue theorem,
\[
  \int_{C_R} f(z)^p\,dz = 2\pi i \sum_{\substack{n \geq 1 \\ |z_n| < R}} \operatorname{Res}_{z_n} f^p.
\]
 
\subsection{Contour Definition}
 The contour consists of two paths $\gamma_1 : t \; , \; t\in [-R , R]$ and $\gamma_2 : R e^{i\theta} , \; \theta\in [0,\pi]$.

\[
 \int_{C_R}(f(z))^p dz =  \int_{-R}^{R}\left(\frac{\rho^2 + 1}{\rho^2 t^2+ 4 \sin^2(\frac{t}{2})} - \frac{1}{t^2} \right)^p dt 
+ 
\]
\[ \int_{0}^{\pi} i\left(\frac{\rho^2 + 1}{\rho^2 R^2 e^{2i \theta}+ 4 \sin^2\left(\frac{Re^{i\theta}}{2}\right)} - \frac{1}{R^2 e^{2i\theta}}\right)^p R e^{i\theta} d\theta\]
\[
 \int_{C_R}f(z)^p\; dz = I_1+I_2
\]

\[ \int_{C_R}f(z)^p \; dz = \int_{\gamma_1}f(z)^p \; dz \; + \; \int_{\gamma_2}f(z)^p \; dz\]

We know that $| z_n| < R$ , $2 n \pi(1 + i \rho) < R$
\[ n < \left|\frac{R}{2\pi(1+i\rho)}\right|\]
for $\rho\to 0$ from the Archimedean principle we have a natural number $N < \frac{R}{2\pi}  \; , \; N \in\mathbb{N}$
\[\int_{\gamma_1}f(z)^p \; dz +\int_{\gamma_2} f(z)^p 
\;dz= 2\pi i  \sum_{n = 1}^{N} Res_{z_{n^+}}f^p\]
We have defined $R = (2 N + 1)\pi$ , the $N^{th}$ pole $z_N = 2\pi N(1+i\rho)$ has the absolute value $|z_N| = 2 \pi N(1+\rho^2)$ , $R > |z_N|$ means all the poles will lie inside the contour. Taking limit ,
\[\lim_{R \to \infty}\left(\int_{\gamma_1}f(z)^p \; dz +\int_{\gamma_2} f(z)^p 
\;
dz\right) = \lim_{R \to \infty}2 \pi i \sum_{n = 1}^{N} Res_{z_{n^+}}f^p\]
\[\lim_{R \to \infty}\int_{-R}^{R}\left(\frac{\rho^2 + 1}{\rho^2 z^2+ 4 \sin^2(\frac{z}{2})} - \frac{1}{z^2} \right)^p \; dz\]
 \[\lim_{R \to \infty}\int_{\gamma_1}f(z)^p \; dz =\int_{-\infty}^{\infty}\left(\frac{\rho^2 + 1}{\rho^2 z^2+ 4 \sin^2(\frac{z}{2})} - \frac{1}{z^2} \right)^p \; dz \]
 \[\lim_{R \to \infty}\int_{0}^{\pi} i\left(\frac{\rho^2 + 1}{\rho^2 R^2 e^{2i \theta}+ 4 \sin^2\left(\frac{Re^{i\theta}}{2}\right)} - \frac{1}{R^2 e^{2i\theta}}\right)^p R e^{i\theta} d\theta\] \[\lim_{R \to \infty}\left|\int_{\gamma_2} f(z)^p 
\;
dz \right | \leq  \lim_{R \to \infty}\sup_{\theta\in[0,\pi]}|(f(Re^{i\theta})^p|. \pi R\]
\[f(Re^{i\theta})^p = i\left(\frac{\rho^2 + 1}{\rho^2 R^2 e^{2i \theta}+ 4 \sin^2\left(\frac{Re^{i\theta}}{2}\right)} - \frac{1}{R^2 e^{2i\theta}}\right)^p \]
\[= \frac{i}{R^{2p}}\left(\frac{\rho^2 + 1}{\rho^2  e^{2i \theta}+ \frac{4\sin^2\left(\frac{Re^{i\theta}}{2}\right)}{R^2} } - \frac{1}{ e^{2 i\theta}}\right)^p \]
\[ \lim_{R \to \infty}sup_{\theta\in[0,\pi]}(O(R^{-2p})). \pi R\]
\[\lim_{R \to \infty}\left|\int_{\gamma_2} f(z)^p 
\;
dz \right | \leq 0\]
hence ,
\[\int_{-\infty}^{\infty} M_\rho(t)^p \, dt = \lim_{R \to \infty}2 \pi i \sum_{n = 1}^{N} Res_{z_{n^+}}f^p\]
from the bound we created earlier we can state that N is related to R as $N(R)$but here all the poles will always lie inside the contour.
\begin{equation}\label{eq:residuesum}
    \lim_{R \to \infty}2 \pi i \sum_{n = 1}^{N(R)} Res_{z_{n^+}}f^p
\end{equation}

%------------------------------------------------------------
\section{Laurent Expansion and Residue Computation}\label{sec:residues}
 
Fix $n \geq 1$ and set $w = z - z_{n^+}$ so that the pole is at $w = 0$. We expand $f(z)^p$ in a Laurent series about $w = 0$ and extract the $w^{-1}$ coefficient.
 
\subsection{Expansion of the Denominator}
 
Writing $z = w + z_n$ with $z_n = 2n\pi(1+i\rho)$ and using $\sin\!\tfrac{z}{2} = \sin\!\tfrac{w+z_{n^+}}{2}$, the denominator of $f(z)$ expands near $w = 0$ as follows.
After substituting the angle addition formula and using $\sin(n\pi) = 0$, $\cos(n\pi) = (-1)^n$, and $\sinh(n\pi\rho) \sim n\pi\rho$ for $\rho \to 0$, the dominant terms are
\[
  \rho^2 z^2 + 4\sin^2\!\frac{z}{2} \;\sim\; 4\sin^2\!\frac{w}{2} + 4\pi ni\rho\sin w.
\]
Applying the Taylor approximations $\sin\!\tfrac{w}{2} \sim \tfrac{w}{2}$ and $\sin w \sim w$ near $w = 0$:
\[
  \sim \; w^2 + 4\pi ni\rho\, w \;=\; w\bigl(w + 4\pi ni\rho\bigr).
\]
 
\subsection{Expansion of the Numerator}

The numerator of $f(z)$ at $z = z_n$ is
\[
C := \frac{z_{n^+}^2 - 4\sin^2\left(\frac{z_{n^+}}{2}\right)}{z_n^2}
= 1 - \frac{4\sin^2\left(\frac{z_{n^+}}{2}\right)}{z_{n^+}^2}.
\]

Using the identity
\[
\sin\left(\frac{z_{n^+}}{2}\right)
= i(-1)^n \sinh(\pi n \rho),
\]
we obtain
\[
\sin^2\left(\frac{z_{n^+}}{2}\right)
= -\sinh^2(\pi n \rho).
\]

Hence,
\[
C = 1 + \frac{4\sinh^2(\pi n \rho)}{z_{n^+}^2}.
\]

Now using
\[
z_{n^+} = 2\pi n (1 + i\rho),
\quad \text{so that} \quad
z_{n^+}^2 = 4\pi^2 n^2 (1 + i\rho)^2 = 4\pi^2 n^2 (1 + 2i\rho - \rho^2),
\]
we obtain
\[
C = 1 + \frac{4\sinh^2(\pi n \rho)}{4\pi^2 n^2 (1 + 2i\rho - \rho^2)}
= 1 + \frac{\sinh^2(\pi n \rho)}{\pi^2 n^2 (1 + 2i\rho - \rho^2)}.
\]

Expanding $\sinh(\pi n \rho)$ for small $\rho$,
\[
\sinh(\pi n \rho) = \pi n \rho + O(\rho^3),
\]
so that
\[
\sinh^2(\pi n \rho) = \pi^2 n^2 \rho^2 + O(\rho^4).
\]

Therefore,
\[
C = 1 + \frac{\pi^2 n^2 \rho^2 + O(\rho^4)}{\pi^2 n^2 (1 + 2i\rho - \rho^2)}
= 1 + \frac{\rho^2}{1 + 2i\rho - \rho^2} + O(\rho^4).
\]

Finally, expanding the denominator,
\[
\frac{1}{1 + 2i\rho - \rho^2}
= 1 - 2i\rho + O(\rho^2),
\]
we obtain
\[
C = 1 + \rho^2(1 - 2i\rho + O(\rho^2)) + O(\rho^4)
= 1 + O(\rho^2).
\]

In particular,
\[
C \to 1 \quad \text{as } \rho \to 0.
\]
 
{Expansion and computing residue of $f^p$}
 
Near $w = 0$, combining numerator and denominator:
\[
  f(z) \;\sim\; \frac{C}{w \cdot (4\pi ni\rho + w + O(w^2))}.
\]
\[
\operatorname{Res}_{w=0} \frac{C^p}{w^p (w + 4\pi n i \rho)^p}
=
\frac{1}{(p-1)!}
\left.
\frac{d^{p-1}}{dw^{p-1}}
\left[
\frac{C^p}{(w + 4\pi n i \rho)^p}
\right]
\right|_{w=0}
\]
\[
=
\frac{C^p}{(p-1)!}
\cdot
(-1)^{p-1}
\frac{(2p-2)!}{(p-1)!}
\cdot
\frac{1}{(4\pi n i \rho)^{2p-1}}
\]
\[
\operatorname{Res}_{w=0}
=
(-1)^{p-1} C^p
\frac{(2p-2)!}{((p-1)!)^2}
\cdot
\frac{1}{(4\pi n i \rho)^{2p-1}}
\]
\[
\operatorname{Res}_{w=0}
=
(-1)^{p-1} C^p
\binom{2p-2}{p-1}
\frac{1}{(4\pi n i \rho)^{2p-1}}
\]
Since $C = 1 + O(\rho^2)$, we have $C^p = 1 + O(\rho^2)$, 
which does not affect the leading order of the residue.
\begin{equation}\label{eq:residue11}
\operatorname{Res}_{w=0}
=
(-1)^{p-1} 
\binom{2p-2}{p-1}
\frac{1}{(4\pi n i \rho)^{2p-1}}
\end{equation}
remember from the $\eqref{eq:poles}$ we have the symmetric poles which contribute to the residue symmetrically. hence the final residue would be
\begin{equation}\label{eq:residue1}
\operatorname{Res}_{w=0}
=
(-1)^{p-1} 
\binom{2p-2}{p-1}
\frac{2}{(4\pi n i \rho)^{2p-1}}
\end{equation}
 
%------------------------------------------------------------
\section{Summation and Final Asymptotics}\label{sec:final}
 
Substituting~\eqref{eq:residue1} into~\eqref{eq:residuesum} and summing over $n \geq 1$:
\[\int_{-\infty}^{\infty} M_\rho(t)^p dt \sim \lim_{R \to \infty}2 \pi i \sum_{n = 1}^{N(R)} Res_{z_{n^+}}f^p\]
\[
\int_{-\infty}^{\infty} M_\rho(t)^p dt 
\sim (-1)^{p-1} 
\binom{2p-2}{p-1}
\frac{4\pi i}{(4\pi  i \rho)^{2p-1}}
\lim_{R \to \infty}\sum_{n=1}^{N(R)} \frac{1}{n^{2p-1}}.
\]

The leading order behavior is governed by $n\lesssim\frac{1}{\rho}$ , as stated earlier  in \eqref{sec:contour} , for $2p-1 > 1$ , the series converges , the tail for $n > \frac{1}{\rho}$ contributes negligibly because
\[\sum_{n >\frac{1}{\rho}} \frac{1}{n^{2p-1}} \leq \frac{1}{N^{2p-1}}\] where $N \sim \frac{1}{\rho}$ so , 
\[\sum_{n> \frac{1}{\rho}}\frac{1}{n^{2p-1}} \leq O(\rho^{2p-1})\]
Hence the tail contribution is negligible relative to the leading order.
\[
\int_{-\infty}^{\infty} M_\rho(t)^p dt 
\sim (-1)^{p-1} 
\binom{2p-2}{p-1}
\frac{4 \pi i}{(4\pi  i \rho)^{2p-1}} 
\sum_{n=1}^{\infty} \frac{1}{n^{2p-1}}.
\]
Thus,
\[
\int_{-\infty}^{\infty} M_\rho(t)^p dt 
\sim \frac{A_p}{\rho^{2p-1}},
\]
where 
\[
A_p
=
\binom{2p-2}{p-1}
\frac{1}{(4\pi)^{2p-2}}
\zeta(2p-1)
\]
for an explicit constant $A_p$. Taking the $p$-th root:
\begin{equation}\label{eq:main-result}
  \boxed{I_p(\rho) = \left(\int_{-\infty}^{\infty} M_\rho(t)^p\,dt\right)^{1/p} \sim C_p\,\rho^{-(2-1/p)}}
\end{equation}
since $\tfrac{2p-1}{p} = 2 - \tfrac{1}{p}$. This completes the proof of Theorem~\ref{thm:main}. \qed
\begin{theorem}[Coiling Barrier]\label{thm:barrier}
Let $p > 1$ and $E < \infty$. Then any helix configuration with $I_p(\rho) \leq E$ 
must satisfy
\[
  \rho \geq c_p \, E^{-1/(2-1/p)}
\]
for an explicit constant $c_p > 0$ depending only on $p$. In particular, no sequence 
of configurations with uniformly bounded $L^p$ energy can have coiling pitch 
$\rho \to 0$. The helix cannot be coiled infinitely tightly within any finite-energy 
class.
\end{theorem}

\begin{proof}
By Theorem~\ref{thm:main}, there exists $\rho_0 > 0$ and $C_p > 0$ such that 
$I_p(\rho) \geq \tfrac{1}{2}C_p\,\rho^{-(2-1/p)}$ for all $\rho \leq \rho_0$.
If $\rho > \rho_0$ the bound $\rho \geq c_p E^{-1/(2-1/p)}$ holds trivially for 
all sufficiently large $E$ by taking $c_p$ small enough. For $\rho \leq \rho_0$, 
the condition $I_p(\rho) \leq E$ gives
\[
  \tfrac{1}{2}C_p\,\rho^{-(2-1/p)} \leq E,
\]
which rearranges to $\rho \geq \bigl(\tfrac{C_p}{2E}\bigr)^{1/(2-1/p)} 
=: c_p\,E^{-1/(2-1/p)}$. \qedhere
\end{proof}

%------------------------------------------------------------
\section{Extension to Real $p$}
\begin{theorem}
    For any real $p > 1$, the $L^p$ integral of the M\"obius energy density satisfies the exact asymptotic scaling:
   \[
  I_p(\rho) \sim C_p \, \rho^{-(2 - 1/p)}, \qquad{p \in \mathbb{R}}
\]
This establishes that the scaling exponent proven in Theorem 1 for integers extends continuously to all real p $>$ 1.

\end{theorem}
\begin{proof}
   Taking $p \in \mathbb{R}$ be any arbitrary p $>$ 1 , where $p_0 = \lfloor p \rfloor$ and $p_1 = \lceil p\rceil$ from \eqref{eq:Ip} we can write 
    \[|| I_{p_0}|| \leq C_p\,\rho^{-(2-1/p_0)} \]
    \[|| I_{p_1}|| \leq C_p\,\rho^{-(2-1/p_1)} \]
    we have $p_0 < p < p_1$ using the log convexity if $L^p $ spaces we can write 
    \[ \frac{1}{p} = \frac{1 -\theta}{p_0} + \frac{\theta}{p_1}\]
    where $\theta \in (0,1)$
    using the Hölder's inequality we have 
    \[ || I_{p}||\leq|| I_{p_0}||^{(1-\theta)} \times || I_{p_1}||^\theta\]
    from here we have 
    \[|| I_{p}||\leq C_p \rho^{-(2-1/p_0)(1-\theta)} \rho^{-(2-1/p_1)\theta}  \]
    \[|| I_{p}||\leq C_p \rho^{-(2+ 2\theta-2\theta - \frac{1-\theta}{p_0}- \frac{\theta}{p_1})}  \]
    \[|| I_{p}||\leq C_p \rho^{-(2 - 1/p)} \]
    we have established an upper bound of the energy functional for the real p's. Examining the energy functional 
    \begin{figure}[h]
        \centering
        \includegraphics[width=0.5\linewidth]{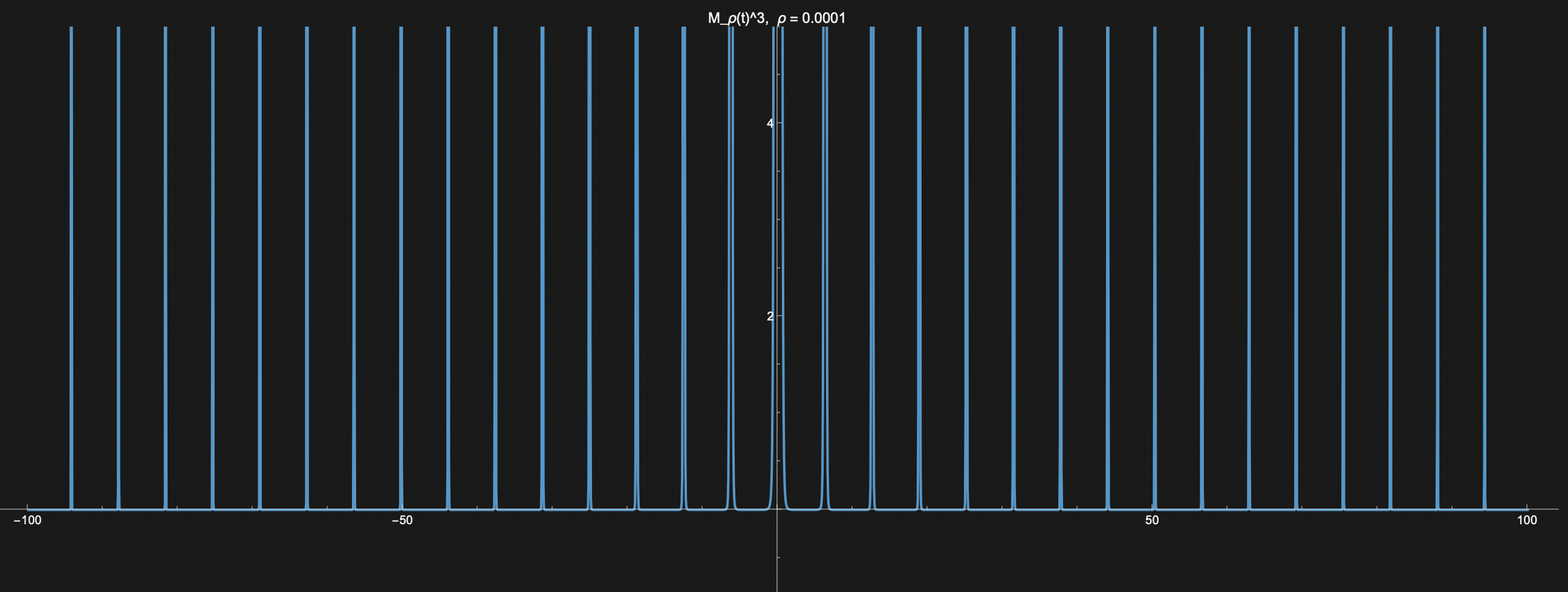}
        \caption{Plot of $M_\rho(t)^p$ for $p=3$ as a function of $t$, showing 
the periodic peak structure used in the lower bound construction.}
        \label{fig:placeholder}
    \end{figure}
    as $\rho \to 0^+$, we observe a periodic 
structure illustrated in Figure~2.

To construct a matching lower bound, we localize the integral around the poles (the peaks in the plot). As $\rho \to 0^+$, the dominant contributions occur near $a_n = 2\pi n$. We partition the domain by considering intervals $[a_n - \pi, a_n + \pi]$ around each pole. 

Shifting the integration variable by $t \mapsto t + a_n$, we define the localized function on $[-\pi, \pi]$:
\[
f_n(t) := M_\rho(t + a_n) = \frac{\rho^2+1}{\rho^2(t+a_n)^2 + 4\sin^2(t/2)} - \frac{1}{(t+a_n)^2}
\]
By the symmetry of the poles across the real axis, the total integral is bounded below by the sum over these localized positive periods:
\[
\int_{-\infty}^{\infty} M_\rho(t)^p \, dt \ge 2 \sum_{n=1}^{\infty} \int_{-\pi}^{\pi} f_n(t)^p \, dt
\]
On the interval $[-\pi, \pi]$, we apply the standard trigonometric bound $c t^2 \le 4\sin^2(t/2) \le t^2$. Substituting this bound allows us to approximate the integral over the shifted domain. For a given $n$, the integration over the dominant peak scales as:
\[
\int_{-\pi}^{\pi} f_n(t)^p \, dt = \int_{-\pi}^{\pi} \left(\frac{\rho^2+1}{\rho^2(t+a_n)^2 + 4\sin^2(t/2)} - \frac{1}{(t+a_n)^2}\right)^p \, dt
\]
\[
\int_{-\pi}^{\pi} f_n(t)^p \, dt  = \int_{-\pi}^{\pi} \left(\frac{\rho^2+1}{\rho^2(t+a_n)^2 + t^2} - \frac{1}{(t+a_n)^2}\right)^p \, dt
\]
Near each pole $a\_n = 2\pi n$, the localized function simplifies as follows. Setting $a = t + a_n$, we combine over a common denominator:
\[f_n(t) = \frac{a_n(a_n + 2t)}
         {(t+a_n)^2\bigl(\rho^2(t+a_n)^2 + t^2\bigr)}\]
         Since $|t| < \pi < 2n\pi$ 
         \[
         f_n(t) \sim \frac{1}{(\rho^2a_n^2+t^2)}
         \]
\begin{equation}
    \int_{-\pi}^{\pi} f_n(t)^p \, dt 
    \sim 
    \int_{-\pi}^{\pi} \frac{1}{(\rho^2 a_n^2 + t^2)^p} \, dt
\end{equation}

For each fixed $n \geq 1$, as $\rho \to 0$ the upper limit 
$\pi/\rho a_n \to \infty$, so:
\begin{equation}
   \frac{1}{(a_n\rho)^{2p-1}} \int_{-\pi/\rho a_n}^{\pi/\rho a_n} \frac{du}{(1+u^2)^p} 
    \longrightarrow   \int_{-\infty}^{\infty} \frac{du}{(1+u^2)^p} 
    =  \frac{1}{(a_n\rho)^{2p-1}} \frac{\sqrt{\pi}\,\Gamma\!\left(p - \tfrac{1}{2}\right)}{\Gamma(p)}
    =  \frac{1}{(a_n\rho)^{2p-1}} \frac{\pi}{4^{p-1}}\binom{2p-2}{p-1}
\end{equation}

Substituting $a_n = 2\pi n$ and factoring out the $\rho$ terms, 
we sum these localized bounds over all $n \geq 1$:
\begin{equation}
    \int_{-\infty}^{\infty} M_\rho(t)^p \, dt 
    \geq 
    2\sum_{n=1}^{\infty} C(\rho 2\pi n)^{1-2p} 
    = C'\rho^{1-2p} \sum_{n=1}^{\infty} \frac{1}{n^{2p-1}}
\end{equation}
Where C' is  the same as $ A_p\sim \eqref{sec:final}$ we computed.

Because $p > 1$, we have $2p - 1 > 1$, meaning the infinite series 
converges to the Riemann zeta function $\zeta(2p-1)$. This leaves us 
with a bounded constant multiplied by the scaling factor:
\begin{equation}
    \int_{-\infty}^{\infty} M_\rho(t)^p \, dt \geq c\rho^{1-2p}
\end{equation}

Taking the $p$-th root of both sides gives the exact lower bound 
for our norm:
\begin{equation}
    I_p(\rho) \geq c^{1/p}(\rho^{1-2p})^{1/p} 
    = c'\rho^{\frac{1}{p}-2} 
    = c'\rho^{-(2-1/p)}
\end{equation}

Combining the arguments from both sides completes the proof. \qed

\end{proof}
%------------------------------------------------------------
\section{Borderline Case p = 1}
 
For $p = 1$ the residue sum reduces to
\[
  \sum_{n=1}^{N} \frac{1}{n} \;\sim\; \log N,
\]
as we have stated, the earlier effective upper cutoff in $n$ is $N \sim 1/\rho$ (poles with $\operatorname{Im}(z_n) = 2n\pi\rho$ leave the width strip $O(1)$ when $n\rho \gtrsim 1$) The effective cutoff arises because poles with imaginary part 
$\operatorname{Im}(z_n) = 2\pi n \rho$ leave the strip $O(1)$ 
when $n\rho \gtrsim 1$, resulting in $N \sim 1/\rho$. Thus we obtain:
\begin{equation}
  I_1(\rho) \sim \frac{\log(1/\rho)}{\rho}.
\end{equation}
This logarithmic correction distinguishes the $p = 1$ case from the clean power law for $p > 1$.
 
%------------------------------------------------------------

\section{Discussion}
 
The main result (Theorem~\ref{thm:main}) establishes a clean family of power laws parameterized by $p$:
\[
  I_p(\rho) \sim C_p \rho^{-(2-1/p)}, \quad p > 1, \; p \in \mathbb{R}.
\]
Several features are worth noting.
\begin{itemize}
\item \textbf{Connection to Lipton's result.} Lipton \cite{lipton2026} 
establishes the asymptotic $I_1(\rho) \sim \log(1/\rho)/\rho$ as 
$\rho \to 0$ for the $L^1$ (arclength-rescaled) M\"obius energy density 
of a helix. Our result generalizes this to the full $L^p$ family: as 
$\rho \to 0$,
\[
I_p(\rho) \sim C_p \rho^{-(2-1/p)}, \quad p > 1.
\]
The exponent $2 - 1/p$ interpolates between the $p=1$ logarithmically-corrected blowup and the $p \to \infty$ quadratic blowup $\|M_\rho\|_{L^\infty} \sim \rho^{-2}$.

\item \textbf{Coiling barrier.} The energy blowup gives a quantitative lower 
bound on the pitch of any finite-energy configuration; see 
Theorem~\ref{thm:barrier}.
 
  \item \textbf{Interpolation.} As $p \to 1^+$ the exponent 
$2 - 1/p \to 1$, and the power law transitions to the logarithmically 
corrected $I_1 \sim \log(1/\rho)/\rho$, consistent with the harmonic 
series divergence in Lipton's proof \cite{lipton2026}.
 
  \item \textbf{Large $p$.} As $p \to \infty$, the exponent $2 - 1/p \to 2$, suggesting $\|M_\rho\|_{L^\infty} \sim \rho^{-2}$, consistent with the pole heights $|f(z_n)| \sim \rho^{-2}$ for $n = O(1)$.
\begin{table}[H]
\centering
\caption{Summary of Coefficients(for $\rho$ = $10^{-3}$)}
\begin{tabular}{@{}llll@{}}
\toprule
$p$ & $I_p$ (numerical) &$C_p \times\rho^{exp}$  & Error (\%) \\ \midrule
2   & 3901.65809            & 3901.82374            & 0.0042 \%    \\
3   & 6295.35922            & 6295.35315            & 0.0001\%    \\
10  & 15491.88126           & 15491.86926           & 0.0001\%    \\
2.5 & 5093.81276            & 5093.80712            & 0.0001\%    \\
4.2 &8845.88319             & 8845.87613            & 0.0001\%    \\\bottomrule
\end{tabular}
\end{table}
  \item   \textbf{Exponent Summary and Multi-$p$ Analysis}
The power law $I_p(\rho) \sim \rho^{-(2-1/p)}$ is verified via log-log regression across multiple orders of magnitude.
\end{itemize}

\begin{table}[H]
\centering
\caption{Summary of Scaling Exponents}
\begin{tabular}{@{}llll@{}}
\toprule
$p$ & Theoretical Exponent & Numerical Exponent & Error (\%) \\ \midrule
2   & -1.5000              & -1.4998            & 0.013\%    \\
3   & -1.6667              & -1.6666            & 0.006\%    \\
4   & -1.7500              & -1.7500            & 0.000\%    \\
10  & -1.9000              & -1.9000            & 0.000\%    \\ 
2.5 & -1.6000              & -1.5981            & 0.1170\%     \\
4.2 & -1.7619              & -1.7607            &0.0634\%     \\\bottomrule
\end{tabular}
\end{table}
\section{Acknowledgment}
The author would like to thank Atreyee Bhattacharya, Ajit Bhand, Armin Schikorra
 for their guidance, helpful discussions, and support 
during this work. The author also thanks Max Lipton  for 
introducing this problem, for helpful discussions throughout, Jun O'Hara for guidance and for facilitating the arXiv endorsement, and Yasuhiko Asao for kindly providing the endorsement. The author will be forever grateful to Rashmi for her constant support throughout the work.

\end{document}